\documentclass[review]{elsarticle}
\usepackage{fullpage,wrapfig}
\journal{ }
\usepackage{amsmath,amssymb,amsthm,multirow,algorithm}
\usepackage[noend]{algpseudocode}
\makeatletter
\def\BState{\State\hskip-\ALG@thistlm}
\makeatother

\newtheorem{theorem}{Theorem}
\newtheorem{ThmObs}{Observation}
\newtheorem{ThmDef}{Definition}

\begin{document}
\begin{frontmatter}
\title{Disaggregated Benders Decomposition for solving a Network Maintenance Scheduling Problem.}
\author{Robin H. Pearce \& Michael Forbes}
\address{School of Mathematics and Physics, University of Queensland, Australia}

\begin{abstract}
We consider a problem concerning a network and a set of maintenance requests to be undertaken. The aim is to schedule the maintenance in such a way as to minimise the impact on the total throughput of the network. We apply disaggregated Benders decomposition using lazy constraints to solve the problem to optimality, as well as explore the strengths and weaknesses of the technique. We prove that our Benders cuts are Pareto-optimal. Solutions to the LP relaxation also provide further valid inequalities to reduce total solve time. We implement these techniques on simulated data presented in previous papers, and compare our solution technique to previous methods and a direct MIP formulation. We prove optimality in many problem instances that have not previously been proven.
\end{abstract}

\end{frontmatter}

\section{Introduction}
Network design and scheduling problems are an important area of study, particularly as they have widespread practical applications. Examples of these problems include minimising the cost of maintaining a network \cite{Boland2013}, restoring a damaged network \cite{Nurre2012} or extending an existing network \cite{Tang2013}. In practice, networks are often large, and optimising their design can be difficult and time-consuming. Industry is always interested in any improvements to operations that result in reduced costs.

Benders decomposition is a powerful technique for breaking a difficult mixed integer program (MIP) into smaller, easier to solve problems \cite{Benders1962}. It has been successfully applied to a number of problems, particularly network design and facility location problems. This technique is especially powerful when the sub-problems can be disaggregated to allow us to add stronger disaggregated Benders cuts. This can make their solutions easier to obtain. Magnanti and Wong (1984) \cite{Magnanti1984} showed that the use of Pareto-optimal cuts with Benders decomposition can improve convergence time by up to 50 times over other Benders cuts. In 2008 Camargo, Miranda and Luna \cite{Camargo2008} applied disaggregated Benders decomposition to the design of hub-and-spoke networks. Tang and Saharidis \cite{Tang2013} used disaggregated Benders decomposition for solving a capacitated facility location problem with existing facilities which could be removed or extended, and Lusby, Muller and Petersen \cite{Lusby2013} used disaggregated Benders decomposition for scheduling the maintenance of power plants in France.

Network design and scheduling problems are perfect candidates for Benders decomposition, as they can be separated into sub- and master problems. The sub-problems for these problems involve computing the maximum flow through the network, and the master problem handles the higher-level design of the network in order to have the greatest flow subject to design constraints. Disaggregated Benders decomposition can be applied if the flow at any time is independent of the flow at other times. This often makes the solution of large network design problems easier.

We consider a problem of performing maintenance on a network to minimise the impact on total flow through the network. Boland et. al. \cite{Boland2013} solved this problem to near-optimality using heuristics. In this paper we show that this problem can be solved quickly using disaggregated Benders decomposition and lazy constraints. The problem formulation we use is similar to that of Boland et. al, with different notation and a necessary change to one constraint. We then decompose the problem into a master problem and a set of sub-problems, and apply disaggregated Benders decomposition to solve it.

The networks considered by Boland et. al. \cite{Boland2013} are rail networks for moving coal from the Hunter Valley to shipping terminals in Newcastle. The arcs of the network represent railway lines and the nodes are junctions, where it is possible to choose which direction to take. In the real network the shipping terminal is comprised of different machines and railways which are also modelled as being part of the same network, and similarly have maintenance jobs assigned to them.

The rest of this section will be a short description of the max-flow min-cut theorem which is useful in solving this problem. In Section 2 we define the problem and present the formulation. Section 3 is where we describe the use of Benders decomposition and lazy constraints to separate and solve the problem. We also look at a scenario which makes this problem more difficult to solve, and prove the pareto optimality of our Benders cuts. In Section 4 we present our results and compare them to those found by Boland et. al. \cite{Boland2013}, as well as to the direct MIP implementation in our version of Gurobi. Section 5 contains concluding remarks.

\subsection{Max-flow min-cut theorem}
The max-flow min-cut theorem was discovered and proven by Elias, Feinstein and Shannon \cite{Elias1956}, and independently also by Ford and Fulkerson \cite{Ford1956}, in 1956. We will use the nomenclature from Elias, Feinstein and Shannon in talking about the max-flow min-cut theorem. 

A \textbf{cut-set} of a network is defined as a set of arcs which, when removed, prevents all flow from the source to the sink. This does not necessarily have to make the graph disconnected, as arcs are allowed to flow backwards, but there will be no complete path flowing forwards from the source to the sink. A \textbf{simple cut-set} is a cut-set which would no longer be a cut-set if any arc was omitted from it. The \textbf{value} of a cut-set is the sum of the capacities of all arcs in the set. A \textbf{minimal cut-set} is the cut-set with the smallest value of all possible cut-sets of the network.

With these definitions, we can state the max-flow min-cut theorem, which we have paraphrased from Elias et. al. \cite{Elias1956}:

\begin{theorem}
	The maximum possible flow from the source to the sink through a network is equal to the minimum value among all simple cut-sets.
\end{theorem}

For proof of this theorem we refer the reader to the original paper by Elias et. al. \cite{Elias1956}. We use this theorem to place bounds on the flow of the network based on the availability of arcs which are in a simple cut-set, especially the minimum cut-set. 

%NOT IN RIGHT PLACE, DOESN'T MAKE SENSE OUT OF CONTEXT.
%Another result from this theorem is that when solving network flow problems without weighted demand, the dual variables associated with the capacity constraints will either be 1 if the arc is in the minimum cut-set, since removing the arc reduces the total flow by the capacity of the arc, or zero otherwise.

\section{Problem definition}
We start with a network $G = (V,A)$ where $V$ is the set of nodes and $A$ is the set of directed arcs. Without loss of generality, assume the network has only one source and one sink, and that there is a directed arc from the sink to the source. This can be achieved by adding two new nodes - one source and one sink - and then adding directed arcs from the new source to the original sources, and from the original sinks to the new sink. The arc that connects the sink to the source now measures the total flow through the network. It has sufficiently high capacity that it does not restrict flow through the network, and is denoted arc $0$.

\subsection*{Sets}
\begin{tabular}{rl}
	\hline $V$ & Set of network nodes\\
	$A$ & Set of network arcs. $A \subseteq V\times V$\\
	$T$ & Set of time periods \\
	$R$ & Set of maintenance requests\\
	$R_a$ & Set of maintenance requests to be performed on arc $a$\\ \hline
\end{tabular}

\subsection*{Data}
\begin{tabular}{rl}
	\hline Cap$_a$ & Capacity of arc $a$\\
	$\delta^-(v)$ & Set of arcs entering node $v$\\
	$\delta^+(v)$ & Set of arcs leaving node $v$\\ 
	Dur$_r$ & Duration of maintenance job $r$\\
	Re$_r$ & Earliest time job $r$ can be started\\
	De$_r$ & Deadline for job $r$\\ \hline
\end{tabular}

\subsection*{Variables}
\begin{tabular}{rl}
	\hline $x_{at}$ & Flow over arc $a$ at time $t$\\
	$y_{at}$ & 1 if arc $a$ is operational at time $t$ and 0 if it is undergoing maintenance\\
	Start$_{rt}$ & 1 if maintenance request $r$ starts at time $t$ and 0 otherwise\\ \hline
\end{tabular}	

We have a set of maintenance requests that must be performed. Each request has a release time Re$_r$, a deadline De$_r$ and a duration Dur$_r$. It is assumed that a work crew will be available for each job regardless of where in the time window the job occurs. If maintenance is being performed on an arc, then the flow along that arc must be 0. If the arc is open, then the flow must not exceed Cap$_a$. As in \cite{Boland2013}, we make the assumption that no two jobs in $R_a$ for any $a$ can overlap. This allows us to strengthen the main scheduling constraint which now says if any maintenance job on a particular arc has started less than Dur$_r$ time periods previously (i.e. it is still running), then the arc is closed, otherwise it is open.

The problem \textit{maximum total flow with flexible arc outages (MaxTFFAO)} is now as follows:

\begin{align}
&\text{Maximise} \sum_{t\in T} x_{0t} \label{MainObj} \tag{MIP-OBJ}\\
&\text{Subject to:} \nonumber \\ 
& \sum_{a\in\delta^-(v)} x_{at} - \sum_{a\in\delta^+(v)} x_{at} = 0 & \forall v \in V, \forall t \in T \label{MainC1}\tag{MIP1}\\
& x_{at} \leq \textup{Cap}_a y_{at} & \forall a\in A, \forall t\in T \label{MainC2}\tag{MIP2}\\
& \sum_{t=Re_r}^{\textup{De}_r-\textup{Dur}_r+1} \text{Start}_{rt} = 1 & \forall r \in R \label{MainC3}\tag{MIP3} \\
& y_{at} + \sum_{r\in R_a}\sum_{t'=\max\{\textup{Re}_{r}, t-\textup{Dur}_r+1\}}^{\min\{t,\textup{De}_r\}} \text{Start}_{rt'} = 1 & \forall a \in A, \forall t \in T \label{MainC4}\tag{MIP4}\\ 
& x_{at} \geq 0,\hspace{5mm} y_{at}\in \{0,1\}, \hspace{5mm}\text{Start}_{rt} \in \{0,1\} & \forall a \in A, \forall t \in T, \forall r \in R \tag{MIP5}
\end{align}

The objective (\ref{MainObj}) is the sum of flow across the arc connecting the sink to the source, which measures the total flow of the network, over the considered time periods. Constraints (\ref{MainC1}-\ref{MainC2}) ensure that flow into and out of a node are the same, and that flow along any arc does not exceed the capacity. Constraint (\ref{MainC3}) ensures that every maintenance job is performed exactly once, and (\ref{MainC4}) is the modified constraint to control the $y_{at}$ variable appropriately.

\section{Disaggregated Benders decomposition and lazy constraints}
In this problem, the continuous variables $x_{at}$, and integer variables $y_{at}$ and Start$_{rt}$, only occur together in one constraint: the capacity constraint (\ref{MainC2}). This allows us to apply Benders decomposition by separating out the continuous variables into a sub-problem, and approximating the solution to the sub-problem with a new variable $\theta$. The result of this is a smaller, more relaxed problem which can be explored faster. 

This decomposition has a natural interpretation. The master problem finds the optimal maintenance schedule given the estimates of the network throughput, and the sub-problems calculate the actual throughput given the network configuration. Because the flow through the network at each time $t \in T$ is independent of the flow at other times, we further break up the problem by disaggregating the sub-problems in $t$, so we solve one sub-problem for each time period.

Traditional Benders decomposition is useful because it results in a smaller, easier to solve master problem. Because the variables associated with the sub-problems are being approximated rather than computed exactly, the master problem is more relaxed than the direct implementation, which could result in a larger branch-and-bound tree to find the optimal solution. While this means more nodes of the branch-and-bound tree need to be explored, the small size of the sub-problems means the nodes can be explored more quickly, and thus the whole problem can be solved more quickly.

Given a feasible solution for $y_{at}$, denoted by $y^*_{at}$, we solve the sub-problems for each time period. Each sub-problem is of the form:

\begin{align}
&\text{Maximise} \hspace{3mm} x_{0t'} \label{SPObj}\tag{SP-OBJ} \\
&\text{Subject to:} \nonumber \\ 
& \sum_{a\in\delta^-(v)} x_{at'} - \sum_{a\in\delta^+(v)} x_{at'} = 0 & \forall v \in V \label{SPC1}\tag{SP1}\\
& x_{at'} \leq Cap_a y^*_{at'} & \forall a\in A \label{SPC2}\tag{SP2}\\
& x_{at'} \geq 0 & \forall a \in A\tag{SP3}
\end{align}

 for every $t' \in T$. These are small linear programs which are easily solved by any good optimisation software package. Having solved the sub-problems, the dual variables associated with constraint (\ref{SPC2}) provide an estimate of the impact on the total flow of switching each arc on or off, which will allow us to add Benders cuts to the master problem. We denote these dual variables as $u_{at'}$, and hence the dual problem will consist of minimising 
 \begin{equation}
 \sum_{a\in\text{A}} \limits Cap_a y^*_{at'}u_{at'}
 \end{equation}
 
 We approximate the solutions to the sub-problems with new variables in the master problem, $\theta_{t'}$. For each $t'\in T$, a valid upper bound on the objective of the dual problem, and hence the objective of the primal problem, is given by:

\begin{align}
\theta_{t'} & \leq \sum_{a\in A}Cap_a y^{*^k}_{at'}u^{*^k}_{at'} + \left(y_{at'}-y^{*^k}_{at'}\right)Cap_a u^{*^k}_{at'} \nonumber \\
\text{Or}& \nonumber \\
\theta_{t'} & \leq \sum_{a\in A}Cap_a y_{at'}u^{*^k}_{at'}\hspace{10mm} \forall k\in \lbrace1 ... K\rbrace, \forall t'\in T \label{BendCut}
\end{align}

where $K$ is the number of times we have added Benders cuts. We now present the master problem for the disaggregated Benders decomposition formulation for the MaxTFFAO problem:

\begin{align}
&\text{Maximise} \sum_{t\in T} \theta_t \label{BMPObj} \tag{MP-OBJ}\\
&\text{Subject to:} \nonumber \\ 
& \theta_t \leq \sum_{a\in A}\textup{Cap}_a y_{at}u^{*^k}_{at} & \forall k\in \lbrace1 ... K\rbrace, \forall t\in T \label{BMPC1}\tag{MP1}\\
& \sum_{t=\textup{Re}_r}^{\textup{De}_r-\textup{Dur}_r+1} \text{Start}_{rt} = 1 & \forall r \in R \label{BMPC2} \tag{MP2}\\
& y_{at} + \sum_{r\in R_a}\sum_{t'=\max\{\textup{Re}_{r}, t-\textup{Dur}_r+1\}}^{\min\{t,\textup{De}_r\}} \text{Start}_{rt'} = 1 & \forall a \in A, \forall t \in T \label{BMPC3}\tag{MP3}\\ 
& \theta_t \geq 0, \hspace{5mm} y_{at}\in \{0,1\}, \hspace{5mm}\text{Start}_{rt} \in \{0,1\} & \forall a \in A, \forall t \in T, \forall r \in R \tag{MP4}
\end{align}

The master problem includes all constraints from the original MIP that do not contain $x_{at}$ variables, replacing them instead with the approximation variables $\theta_t$. An advantage of the disaggregation of the flow problems is that they now only depend on the configuration of arcs $y_{at'}$ for each $t' \in T$. As such, we store solutions to these flow problems, where the key is the vector $(y_{at'} |\forall a \in A)$. If a flow problem with a certain configuration of arcs has already been solved, we recall it from memory rather than calculating it again. The number of times we solve and recall flow problems will be shown in the results section for some cases. There are several other improvements we have made to increase the effectiveness of the solver. 

While exploring the branch-and-bound tree, at each integer solution, the values of $\theta_t$ are checked. If the approximations are incorrect, we add a new Benders cut to the problem, which we add as a lazy constraint. This is the most efficient way of implementing Benders decomposition. Our implementation of Benders decomposition with only what has been described above will be referred to as DBD.

\subsection{Initial cuts}

\begin{figure}
	\includegraphics[width=0.31\textwidth]{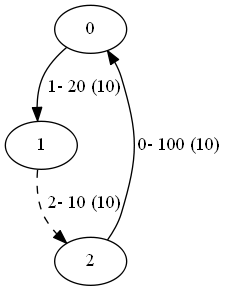}
	\includegraphics[width=0.31\textwidth]{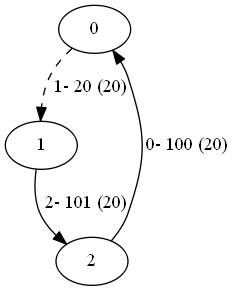}
	\includegraphics[width=0.34\textwidth]{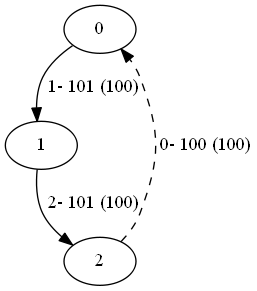}

	\caption{Example of placing cuts on bottlenecks. Numbers next to arcs are of the form: ArcID - capacity (flow). Dashed lines are arcs whose associated dual variables are non-zero.}
	\label{FigICs}
\end{figure}

We begin by considering the case where all arcs are turned on, i.e. $y_{at} = 1 \hspace{2mm}\forall a \in A, \forall t \in T$. This gives us an upper bound on the flow through the network in any case, because turning an arc off cannot possibly increase the total flow. The set of arcs which have a non-zero dual variable associated with their capacity constraints is a ``minimum cut-set'' from the max-flow min-cut theorem \cite{Elias1956}. In other words, they are bottlenecks of the network, since turning any of them off will directly affect the total flow through the network. 

\begin{wrapfigure}{R}{0.42\textwidth}
	\includegraphics[width=0.41\textwidth]{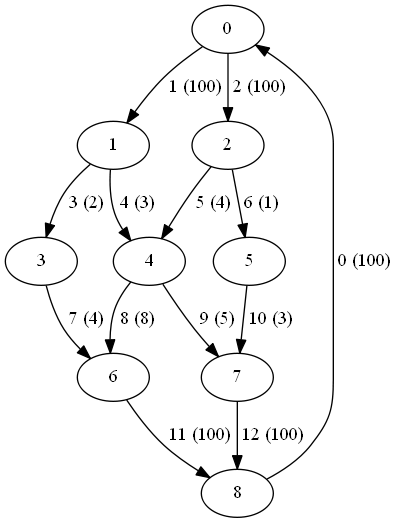}
	\caption{A layered network. The edges are labelled by ArcID (capacity). The two layers have capacities of 10 and 20.}
	\label{FigIC2}
	
\end{wrapfigure}

Consider the trivial case in Figure \ref{FigICs}. To start with, the arc with capacity 10 is the bottleneck. It is the only arc with a non-zero dual variable, so the initial cut will be $\theta_t \leq 10y_{2,t}$ for every $t$. Next, we change the capacity of this arc to be larger than that of the total flow arc (i.e. 100). When we solve the flow problem again, we see that arc 1 will be the bottleneck. Since it is the only arc with a non-zero dual variable, we add another cut $\theta_t \leq 20y_{1,t}$ for every $t$. We increase the capacity of this arc as before, calculate the solution to the new flow problem, and find that the total flow arc is now the bottleneck. When this occurs, we are finished adding initial cuts.

Both of these initial cuts are valid, since turning off either of these arcs will restrict all flow through the network and $\theta_t = 0$. For larger problems, the bottlenecks will consist of multiple arcs, and the cuts provide information about how closing arcs in those bottlenecks affects the flow. 

Consider now the less-trivial example in Figure \ref{FigIC2}. This is a layered network, where the sum of the capacities of arcs 3-6 is 10 and of arcs 7-10 is 20. In this case the initial cut-set will be arcs 3-6 since they form the minimum cut-set. The cuts we add will be:

\begin{equation}
\theta_t \leq 2y_{3,t} + 3y_{4,t} + 4y_{5,t} + y_{6,t} \hspace{10mm} \forall t \in T
\end{equation}

Because this is a minimum cut-set by the max-flow min-cut theorem, cutting any of these arcs will reduce the total flow through the network. When we increase these capacities and solve the maximum flow problem again, we find the second cut-set of arcs 7-10. We then add the cuts:

\begin{equation}
\theta_t \leq 4y_{7,t} + 8y_{8,t} + 5y_{9,t} + 3y_{10,t} \hspace{10mm} \forall t \in T
\end{equation}

which are also valid.

With these initial cuts, we can find a solution to the master problem which satisfies our scheduling constraints. We then solve the problem as before, except now we start with a tighter LP bound. This implementation will be labelled as Pre-cuts.

Another potential improvement is user-suggested heuristics. While the MIP solver is exploring nodes, adding Benders cuts will cut off the current solution because the values of $\theta_t$ are too high. When this happens, we construct a feasible solution with the same values of $y_{at}$ and $Start_{rt}$, but set $\theta_t$ equal to the solutions to the sub-problems that were solved, and suggest this to the solver as a heuristic solution. In our experiments, it has in some cases led to significant jumps in reducing the optimality gap. This implementation is called MAIN.

Finally, we implement what is known as a ``warm start'', where we relax the integrality constraints and run the main algorithm to add Benders cuts as in (\ref{BendCut}). First suggested by McDaniel and Devine in 1977\cite{McDaniel1977}, a warm start can often improve the initial bound of the Benders master problem. We do this repeatedly until the objective value found by the relaxed problem stops decreasing. Once this occurs, we restore the integrality constraints and solve the problem once more. This results in a tighter LP bound and often leads to great improvements in the solve time of the MIP. However, the time it takes to solve the relaxed problem multiple times must be taken into account. The full implementation with all features is called LP-R.

\subsection{Strength of the LP-Relaxation} \label{LPSubSect}
All jobs $r \in R$ can start during the time window [Re$_r$,De$_r$-Dur$_r$+1]. If the size of this window is larger than the duration of the job, then the LP-relaxation of the problem provides a weaker bound. This is because it is possible to fractionally assign values to Start$_{rt}$ and thus have arcs fractionally open for more than the duration of the maintenance. The result of this is a weaker LP bound on the objective value and thus a longer solve time, which applies to all MIP implementations.

Consider the example of a job $r$ where Re$_r = 0$, Dur$_r = 10$ and De$_r = 29$. This job could start at times $t' \in [0,20]$. We are only considering the scheduling constraints (\ref{MainC3}-\ref{MainC4}) here. When the variables $y_{at}$ and Start$_{rt}$ are allowed to be continuous, it is possible for Start$_{rt}$ to take values of $\frac{1}{3}$ at times 0, 10 and 20, and 0 elsewhere. Because in constraint (\ref{MainC4}) we sum the Start$_{rt}$ over values of $t'$ within Dur$_r$ time periods previous to $t$, at every $t \in [$Re$_r,$De$_r]$, 

\begin{align*}
&\sum_{r\in R_a}\sum_{t'=\max\{\textup{Re}_{r}, t-\textup{Dur}_r+1\}}^{\min\{t,\textup{De}_r\}} \text{Start}_{rt'} = \frac{1}{3}
\end{align*}

This implies that

\begin{align*}
y_{at} + \frac{1}{3} &= 1 & \forall a \in A, \forall t \in T \\
y_{at} &= \frac{2}{3} & \forall a \in A, \forall t \in T
\end{align*}

This means that the arc the job is being performed on will be fractionally closed for almost the entire time between Re$_r$ and De$_r$, whereas in the integer program it must be fully closed for Dur$_r$. If closing this arc results in a change in the minimum cut-set of the network, but fractionally closing the arc does not, then the relaxed problem will not properly reflect the impact on the objective value from closing this arc. When exploring the branch-and-bound tree, we give branching priority to the $y_{at}$ over the Start$_{rt}$ variables to allow more balanced branching to occur. That is, the impact of setting a branch variable to 1 is similar to the impact of setting it to 0.

\subsection{Algorithm details}
We have included pseudo-code for our algorithms to give a brief idea of how our implementations are set up. The first is the main procedure, which includes potential calls to the sub-routines PRE-CUTS and LP-RELAX, depending on which implementation is being used. When we talk about building models, we are referring to creating a Model object in Gurobi\cite{gurobi}.

\begin{algorithm}
	\caption{Main Procedure} \label{AlgMP}
	\hspace{5mm} \\
	Initialise information about Network and Jobs: $N, A, Cap_a, R, R_a, T$ \\
	Create empty dictionary Y to hold solutions to Sub-Problems \\
	Build Sub-Model and Master Model \\
	Initialise Sub-Model variables $x_a$ and Master Model variables $y_{at}, Start_{rt}$ \\
	Set $y_{at} = 1$ $\forall a \in A,$ $\forall t \in T$ \\
	OPTIMISE Sub-Model for one time value\\
	Add constraints: $\theta_t \leq x_0$ $\forall t \in T$ \\
	\begin{algorithmic}
		\If {Pre-cuts}
		\State Run procedure PRE-CUTS
		\EndIf
		\If {LP-Relax}
		\State Run procedure LP-RELAX
		\EndIf
	\end{algorithmic}
	Run procedure OPTIMISE MASTER MODEL
\end{algorithm}

Because we have disaggregated the sub-problems in time, and they are all identical, we only need to build one model and use it to solve the flow sub-problems for all time periods. The results of these sub-problems are stored in a dictionary Y. For any time $t'$, ($y_{at'}$ $\forall a \in A$) will be the configuration of arcs of the network, i.e. which arcs are open and which are closed. 

\begin{algorithm}
	\caption{OPTIMISE MASTER MODEL} \label{AlgOMM}
	\begin{algorithmic}
		\State {OPTIMISE Master Model with callback MMCB}
		\State {MMCB:}
		\If {Found new incumbent solution}
		\ForAll {$t' \in T$}
		\State {Retrieve values of $y_{at'}$ and pass to Sub-Model}
		\If {not $y_{at'}$ in Y}
		\State{OPTIMISE Sub-Model}
		\State{Y$[y_{at'}]$ $\gets$ ($x_0$, $x_a$ $\forall a \in A$, $u_a$ $\forall a \in A$)}
		\Else 
		\State {($x_0$, $x_a$ $\forall a \in A$, $u_a$ $\forall a \in A$) $\gets$ Y$[y_{at'}]$}
		\EndIf
		\State {$\bar{\theta_t'} \gets x_0$}
		\If {$\theta_t' > x_0$}
		\State {Add lazy constraint $\theta_t' \leq \sum_{a\in A} \limits Cap_ay_{at'}u_a$}
		\EndIf
		\EndFor
		\If {$\sum_{t \in T} \limits \theta_t > \sum_{t \in T} \limits \bar{\theta_t}$}
		\State {Suggest $\theta_t = \bar{\theta_t}$ $\forall t \in T$ as heuristic solution}
		\EndIf
		\EndIf
	\end{algorithmic}
\end{algorithm}

The configuration $y_{at'}$ is the key to a dictionary entry which holds a tuple. The first value is the total flow through the network and the second is a vector $u_a$ $\forall a \in A$ of the dual variables associated with the capacity constraints of each arc.

\begin{algorithm}
	\caption{PRE-CUTS}\label{AlgPC}
	\begin{algorithmic}
		\State{Retrieve dual variables from Sub-Model $u_a$ $\forall a \in A$} 
		\While {not $u_0 > 0$}
		\State {Add constraints $\theta_t \leq \sum_{a\in A} \limits Cap_ay_{at}u_a$ $\forall t \in T$}
		\ForAll {$ a \in A$}
		\If {$u_a > 0$}
		\State {$Cap_a = Cap_0 + 1$}
		\EndIf
		\EndFor
		\State {OPTIMISE Sub Model}
		\State{Retrieve dual variables from Sub-Model $u_a$ $\forall a \in A$} 			
		\EndWhile
		\State {Reset values of $Cap_a$ $\forall a \in A$}
	\end{algorithmic}
\end{algorithm}

When we need the solution to a sub-problem for a certain configuration of arcs, we first check to see if we have already solved it. If $y_{at'}$ is a valid key to the dictionary we simply recall the tuple stored in that entry. If that particular configuration has not been solved, we pass the values $y_{at'}$ to the sub-problem model and solve the max-flow problem. We then store the results of this in the dictionary under the key $y_{at'}$.

\begin{algorithm}
	\caption{LP-RELAX}\label{AlgLPR}
	\begin{algorithmic}
		\State {Relax integrality constraints for $y_{at}$ and $Start_{rt}$}
		\While {True}
		\State{Run procedure OPTIMISE MASTER MODEL}
		\ForAll {$t' \in T$}
			\State {Retrieve values of $y_{at'}$ $\forall a \in A$ and pass to Sub-Model}
			\State{OPTIMISE Sub-Model}
			\If {$\theta_t' > x_0$}
				\State {Add constraint $\theta_t' \leq \sum_{a\in A} \limits Cap_ay_{at'}u_a$}
			\EndIf
		\EndFor
		\If {Objective has not improved, time limit expired or max iterations reached}
		\State {Exit While}
		\EndIf
		\EndWhile
		\State{Enforce integrality constraints for $y_{at}$ and $Start_{rt}$}
	\end{algorithmic}
\end{algorithm}

The last ``if'' statement of Algorithm \ref{AlgOMM} is the addition we made to the MAIN implementation that is not present in DBD or Pre-cuts. After all new lazy constraints have been added, it is possible that the current values of $\theta_t$ may violate them, so we suggest a heuristic solution to the solver instead. This solution uses the same configuration of arcs $y_{at}$ $\forall a \in A$, $\forall t \in T$, however we set the values of $\theta_t$ to the solutions of the sub-problems, which we know do not violate the new constraints.

\subsection{Proof of Pareto-optimality of Benders cuts}
It has been shown that the use of Pareto-optimal cuts can greatly improve the convergence rate of Benders decomposition \cite{Tang2013,Magnanti1984,Magnanti1981}. Pareto-optimal cuts are especially powerful when there is  degeneracy in the sub-problems of the Benders decomposition, which is the case in network design problems \cite{Magnanti1984}. The definitions of dominating and Pareto-optimal cuts we use come from Magnanti and Wong (1981) \cite{Magnanti1981}, but have been modified to match our problem.

Since these cuts are disaggregated in time, we will omit all $t$ parameters for simplicity. This means we will consider $\theta$ instead of $\theta_t$, and likewise $x_a$, $y_a$, $u_a$. Since our Benders cuts depend only upon our $y$ variables, we write them in a general form $\theta \leq \bar{\theta}^k(y)$, where $\bar{\theta}^k(y) = \sum_{a\in A}\limits \textup{Cap}_a y_{a} u_a^k$.
 
 \begin{ThmDef}
 	A Benders cut $\theta \leq \bar{\theta}^k(y)$ dominates another Benders cut $\theta \leq \bar{\theta}^l(y)$ if $\bar{\theta}^k(y) \leq \bar{\theta}^l(y)$ for all feasible $y$ and is a strict inequality for at least one feasible $y$.
 \end{ThmDef}
 
The contrapositive of this is that if there exists a feasible solution $y^*$ such that $\bar{\theta}^k(y^*) > \bar{\theta}^l(y^*)$, then $\theta \leq \bar{\theta}^k(y)$ does not dominate $\theta \leq \bar{\theta}^l(y)$.
 
\begin{ThmDef}
	A Benders cut $\theta \leq \bar{\theta}^k(y)$ is considered Pareto-optimal if it is not dominated by any other cuts.
\end{ThmDef}
 
 That is to say, if for any other Benders cut $\theta \leq \bar{\theta}^l(y)$ one can find a feasible solution $y^*$ such that $\bar{\theta}^k(y^*) < \bar{\theta}^l(y^*)$, then $\theta \leq \bar{\theta}^k(y)$ is Pareto-optimal.
 
\begin{ThmObs}
 	For a given network configuration $y_{a}$ and its primal and dual flow solutions $x_a$ and $u^k_a$, the set $A^k = \{a\in A|u^k_a > 0\} \subset A$ must constitute a simple cut-set of the original network.
\end{ThmObs}
 
 This comes from the max-flow min-cut theorem. Because our network has uniform demand weighting, we redefine this set as $A^k = \{a\in A|u^k_a = 1\} \subset A$. If the set is not a simple cut-set, then the Benders cut generated by the set will not be Pareto-optimal. If $A^k$ is not a cut-set, then we choose the network configuration $y^*$ where all arcs in $A^k$ are closed and all other arcs are open. Since $A^k$ is not a cut-set, it is still possible for flow between the source and the sink to occur, so $\theta > 0$, however we also have that $\theta \leq \sum_{a \in A^k}\limits \textup{Cap}_a y^*_a = 0$. This means the constraint generated by $A^k$ is invalid.
 
 If the set $A^k$ is a cut-set but not a simple cut-set, then there exists an arc $a' \in A^k$ such that $A^k \setminus \{a'\}$ is still a cut-set. Let $A^l = A^k \setminus \{a'\}$, which means $A^l \cup \{a'\} = A^k$. We now compare the Benders cuts generated by these two sets:
 
 \begin{equation}
 \theta \leq \sum_{a \in A^l} Cap_a y_a < \sum_{a \in A^l} Cap_a y_a + Cap_{a'} y_{a'} = \sum_{a \in A^l \cup \{a'\}} Cap_a y_a = \sum_{a\in A^k} Cap_a y_a
 \end{equation}
  
  So the Benders cut $\theta \leq \bar{\theta}^l(y)$ dominates $\theta \leq \bar{\theta}^k(y)$, and thus the Benders cut generated by $A^k$ cannot be Pareto-optimal. Using this we show that all the Benders cuts we generate are Pareto-optimal.
 
\begin{theorem}
 	Given a simple cut-set $A^k$, the Benders cut $\theta \leq \sum_{a\in A^k} \limits Cap_a y_{a} = \bar{\theta}^k(y)$, is Pareto-optimal.
\end{theorem}

\begin{proof}
	The cut-set $A^k$ does not have to be minimal in the original network. Now, for any other Benders cut $\theta \leq \bar{\theta}^l(y)$, we have another cut-set $A^l = \{a \in A|u^l_a >0\} \subset A$, and $A^k \neq A^l$. If we compare these two Benders cuts, we get
	
	\begin{align}
	\theta \leq \sum_{a \in A^k} Cap_a y_a = \bar{\theta}^k(y) \\
	\theta \leq \sum_{a \in A^l} Cap_a y_a = \bar{\theta}^l(y)
	\end{align}
	
	For $\theta \leq \bar{\theta}^k(y)$ to be Pareto-optimal, we need to find a solution $y^*$ such that $\bar{\theta}^k(y^*) < \bar{\theta}^l(y^*)$. Because $A^k \neq A^l$, we choose an arc $a'$ such that $a' \in A^l$ and $a' \notin A^k$. Now we can take the solution
	
	\begin{equation}
	y^*_a = \begin{cases} 0,&a\in A^k\\ 1,& \text{otherwise}  \end{cases}
	\end{equation}
	
	which is to turn off all arcs in $A^k$ and open all other arcs. Our Benders cuts now look like:
	
	\begin{align}
	\theta \leq \sum_{a \in A^k} Cap_a y^*_a = 0 \\
	\theta \leq \sum_{a \in A^l} Cap_a y^*_a \geq Cap_{a'} > 0
	\end{align}
	
	So $\theta \leq \bar{\theta}^l(y)$ does not dominate $\theta \leq \bar{\theta}^k(y)$. Since $A^l$ is arbitrary, we have that $\theta \leq \bar{\theta}^k(y)$ is Pareto-optimal.♣
\end{proof}

\section{Results and comparison}
We tested our implementation on the same data as Boland et. al \cite{Boland2013}, using several implementations. We started with disaggregated Benders decomposition using lazy constraints (DBD), added the tighter initial cuts (Pre-cuts), suggested new solutions as heuristics (MAIN) and solved the LP relaxation beforehand (LP-R). Our program is implemented in Python 2.7.10 (64-bit) and uses the Gurobi 6.5 (64-bit on 8 threads) solver package \cite{gurobi}, running on a machine with Windows 8.1 Enterprise (64-bit). The machine has an Intel Core i7-3770 (3.40GHz) with 8GB of RAM. 

Data we collected included the total run time of the program, the optimality gap, the number of lazy cuts added, branch-and-bound nodes explored and how many times the sub-problem was solved and recalled from memory. Total run time was recorded as an integer number of seconds. For the LP-R implementation, we also recorded how many times the relaxed problem was solved. 

Because $Cap_a \in \mathbb{N}$ for our data sets, all feasible solutions will have integer objective values. This allows us to set a termination condition for the MIP gap, since an absolute gap of less than one is sufficient to prove optimality. Without this condition, some programs find the optimal solution in less than 10 minutes, and then spend vast amounts of time trying to close the gap by exploring hundreds of thousands of nodes. Since this is unnecessary, we will terminate the program if the absolute gap is less than 0.999.

Comparing our results with those of Boland et. al. is not a simple task. As they used a straight MIP formulation in CPLEX and a number of heuristics, it is difficult to report optimality gaps. For the heuristics, the optimality gap is computed using the best upper bound found by the CPLEX implementation, which is not proving optimality in many cases. This means that their optimality gaps are over-estimates, and their heuristics may be closer to optimality than reported.

\subsection{Simulated Data}
Each of the three constructed data sets from Boland et. al. has eight networks of increasing size, and each network has 10 randomly-generated lists of maintenance requests. For all instances, the number of jobs per arc is between five and 15, and the duration is between 10 and 30 time steps. For the first instance, the number of possible starting times for each job ranges between one and 35, whereas in the second instance set, each job has between 25 and 35 potential start times. The second instance set is thus more difficult to solve in general, because there is a much higher chance of having jobs where the potential starting window is larger than the duration of the job, which causes the problem discussed in section \ref{LPSubSect}. Finally, there is a third instance set where the number of possible starting times for each job is between one and 10. This is an especially easy case as there will almost never be a job with the aforementioned problem.

\begin{table}
	\caption{Average, minimum and maximum optimality gap, number of schedules solved to optimality, number of flow problems solved and recalled and number of lazy constraints generated for instance set 1 after 5 minutes}
	\label{TabI1Gap5}
	\begin{tabular*}{\textwidth}{@{\extracolsep{\fill}}cccccccccc}
		\hline & & 1 & 2 & 3 & 4 & 5 & 6 & 7 & 8 \\ 
		\hline \multirow{3}{*}{DBD} & avg. gap (\%) & 0.000 & 0.000 & 0.000 & 0.027 & 0.416 & 0.058 & 0.202 & 0.088 \\ 
		& min gap (\%) & 0.000 & 0.000 & 0.000 & 0.000 & 0.167 & 0.000 & 0.000 & 0.000 \\ 
		& max gap (\%) & 0.000 & 0.000 & 0.000 & 0.101 & 0.583 & 0.180 & 0.843 & 0.757 \\ 
		& num solved & 10 & 10 & 10 & 3 & 0 & 5 & 1 & 7 \\ 
		& \# flow solved & 3114 & 6587 & 3631 & 18141 & 20272 & 19186 & 18885 & 16248 \\ 
		& \# flow recalled & 27757 & 37300 & 14064 & 38396 & 39500 & 43362 & 29798 & 25722 \\ 
		& \# lazy gen. & 3189 & 4348 & 2657 & 4502 & 4964 & 6516 & 4386 & 4281 \\ 
		\hline \multirow{3}{*}{Pre-cuts} & avg. gap (\%) & 0.000 & 0.000 & 0.000 & 0.027 & 0.308 & 0.025 & 0.234 & 0.120 \\ 
		& min gap (\%) & 0.000 & 0.000 & 0.000 & 0.000 & 0.223 & 0.000 & 0.000 & 0.000 \\ 
		& max gap (\%) & 0.000 & 0.000 & 0.000 & 0.085 & 0.368 & 0.098 & 1.474 & 1.162 \\ 
		& num solved & 10 & 10 & 10 & 3 & 0 & 6 & 2 & 7 \\ 
		& \# flow solved & 2352 & 5024 & 2950 & 16015 & 16215 & 13952 & 21081 & 12211 \\ 
		& \# flow recalled & 22124 & 31866 & 13449 & 40018 & 35549 & 34483 & 32914 & 24449 \\ 
		& \# lazy gen. & 1033 & 1583 & 1053 & 1834 & 1743 & 2163 & 1838 & 1341 \\ 
		\hline \multirow{3}{*}{MAIN} & avg. gap (\%) & 0.000 & 0.000 & 0.000 & 0.036 & 0.344 & 0.025 & 0.191 & 0.034 \\ 
		& min gap (\%) & 0.000 & 0.000 & 0.000 & 0.000 & 0.159 & 0.000 & 0.000 & 0.000 \\ 
		& max gap (\%) & 0.000 & 0.000 & 0.000 & 0.130 & 0.465 & 0.097 & 1.151 & 0.277 \\ 
		& num solved & 10 & 10 & 10 & 3 & 0 & 7 & 1 & 6 \\ 
		& \# flow solved & 2352 & 4973 & 2950 & 17227 & 15660 & 13918 & 20346 & 13754 \\ 
		& \# flow recalled & 22324 & 31015 & 13449 & 42412 & 31096 & 34519 & 28337 & 25111 \\ 
		& \# lazy gen. & 1033 & 1541 & 1053 & 1820 & 1768 & 1955 & 1764 & 1413 \\ 
		\hline \multirow{3}{*}{LP-R} & avg. gap (\%) & 0.000 & 0.000 & 0.000 & 0.041 & 0.387 & 0.037 & 0.143 & 0.031 \\ 
		& min gap (\%) & 0.000 & 0.000 & 0.000 & 0.000 & 0.238 & 0.000 & 0.000 & 0.000 \\ 
		& max gap (\%) & 0.000 & 0.000 & 0.000 & 0.101 & 0.543 & 0.165 & 0.609 & 0.226 \\ 
		& num solved & 10 & 10 & 10 & 3 & 0 & 6 & 2 & 7 \\ 
		& \# flow solved & 2213 & 4243 & 2424 & 14297 & 13419 & 13887 & 15000 & 13200 \\ 
		& \# flow recalled & 21562 & 27245 & 11572 & 35927 & 26629 & 35149 & 23766 & 21957 \\ 
		& \# lazy gen. & 781 & 1071 & 700 & 1507 & 1358 & 1751 & 1228 & 1043 \\ 
		& avg. LP solves & 3.5 & 3.9 & 4.0 & 3.9 & 4.5 & 4.3 & 3.7 & 3.4 \\ 
		\hline 
	\end{tabular*} 
\end{table}

Table \ref{TabI1Gap5} shows the average, minimum and maximum optimality gaps achieved while solving the first instance set using our four implementations. It also includes the number of maximum flow sub-problems solved, number of times solutions were recalled from memory, and number of lazy constraints added to the problem. The first three networks were all solved to optimality in less than five minutes, and even a number of the instances for the larger networks found optimal solutions in the time frame. The average optimality gap for MAIN and LP-R always was less than 0.4\%, and the maximum gap for LP-R never exceeded 0.61\%.

\begin{table}
	\caption{Average, minimum and maximum optimality gap, number of schedules solved to optimality, number of flow problems solved and recalled and number of lazy constraints generated for instance set 2 after 5 minutes}
	\label{TabI2Gap5}
	\begin{tabular*}{\textwidth}{@{\extracolsep{\fill}}cccccccccc}
		\hline & & 1 & 2 & 3 & 4 & 5 & 6 & 7 & 8 \\ 
		\hline \multirow{3}{*}{DBD} & avg. gap (\%) & 0.000 & 0.529 & 0.000 & 0.291 & 2.771 & 4.227 & 0.621 & 0.424 \\ 
		& min gap (\%) & 0.000 & 0.053 & 0.000 & 0.000 & 0.158 & 2.014 & 0.006 & 0.000 \\ 
		& max gap (\%) & 0.000 & 1.586 & 0.000 & 0.913 & 5.609 & 9.121 & 2.800 & 1.135 \\ 
		& num solved & 10 & 0 & 10 & 2 & 0 & 0 & 0 & 3 \\ 
		& \# flow solved & 5932 & 13573 & 4353 & 17512 & 20770 & 28459 & 26945 & 24295 \\ 
		& \# flow recalled & 43280 & 54534 & 13340 & 42770 & 30478 & 48951 & 37984 & 26501 \\ 
		& \# lazy gen. & 3944 & 5984 & 2800 & 4572 & 5524 & 7879 & 4794 & 4965 \\ 
		\hline \multirow{3}{*}{Pre-cuts} & avg. gap (\%) & 0.000 & 0.442 & 0.000 & 0.622 & 3.300 & 2.819 & 0.994 & 0.302 \\ 
		& min gap (\%) & 0.000 & 0.116 & 0.000 & 0.000 & 0.316 & 0.482 & 0.002 & 0.000 \\ 
		& max gap (\%) & 0.000 & 1.166 & 0.000 & 2.630 & 7.417 & 6.165 & 3.093 & 1.516 \\ 
		& num solved & 10 & 0 & 10 & 2 & 0 & 0 & 0 & 6 \\ 
		& \# flow solved & 3853 & 11808 & 2860 & 13058 & 14770 & 16064 & 21034 & 16490 \\ 
		& \# flow recalled & 28257 & 53700 & 13334 & 36906 & 26575 & 33510 & 31870 & 27193 \\ 
		& \# lazy gen. & 1357 & 2693 & 966 & 1626 & 1587 & 2075 & 1842 & 1757 \\ 
		\hline \multirow{3}{*}{MAIN} & avg. gap (\%) & 0.000 & 0.486 & 0.000 & 0.836 & 2.646 & 2.819 & 0.638 & 0.308 \\ 
		& min gap (\%) & 0.000 & 0.092 & 0.000 & 0.000 & 0.089 & 0.695 & 0.072 & 0.000 \\ 
		& max gap (\%) & 0.000 & 1.368 & 0.000 & 2.615 & 5.123 & 6.165 & 1.259 & 1.516 \\ 
		& num solved & 10 & 0 & 10 & 2 & 0 & 0 & 0 & 6 \\ 
		& \# flow solved & 3687 & 11087 & 2860 & 12434 & 11740 & 14687 & 19735 & 17752 \\ 
		& \# flow recalled & 27921 & 51920 & 13334 & 30628 & 22189 & 29880 & 25855 & 23726 \\ 
		& \# lazy gen. & 1298 & 2462 & 966 & 1728 & 1469 & 2080 & 1807 & 1851 \\ 
		\hline \multirow{3}{*}{LP-R} & avg. gap (\%) & 0.000 & 0.508 & 0.000 & 0.723 & 2.826 & 4.130 & 0.561 & 0.337 \\ 
		& min gap (\%) & 0.000 & 0.078 & 0.000 & 0.000 & 0.201 & 0.602 & 0.122 & 0.000 \\ 
		& max gap (\%) & 0.000 & 1.598 & 0.000 & 3.327 & 6.867 & 9.209 & 1.455 & 2.369 \\ 
		& num solved & 10 & 0 & 10 & 2 & 0 & 0 & 0 & 6 \\ 
		& \# flow solved & 3683 & 10773 & 2494 & 14180 & 10696 & 14411 & 18712 & 13693 \\ 
		& \# flow recalled & 28336 & 52240 & 11400 & 29075 & 19633 & 26552 & 24373 & 21373 \\ 
		& \# lazy gen. & 1075 & 2202 & 746 & 1637 & 1387 & 1852 & 1672 & 1475 \\ 
		& avg. LP solves & 2.0 & 2.0 & 2.0 & 2.0 & 2.0 & 2.0 & 2.0 & 2.0 \\ 
		\hline 
	\end{tabular*}
\end{table}

The second instance set was more difficult to solve, as expected. Table \ref{TabI2Gap5} shows that in many cases the problems did not solve to optimality within five minutes, though the gap was closed significantly more than in Boland et. al. \cite{Boland2013}. Many instances closed to within 5\% of optimality in the five-minute time window, and the average gap for MAIN was less than 3\% in all cases.

The number of times the relaxed LP was solved is interesting. Our algorithm runs the relaxed LP and adds cuts until the objective function stops decreasing. For all cases in instance set 2, the LP was solved twice, which means there was no improvement after the first set of cuts. This implies that there was no direct benefit to adding cuts based on the solution to the relaxed problem. This can also be seen from the fact that in many cases the average optimality gap of LP-R was similar to or greater than that of MAIN.

We then compared only the MAIN and LP-R programs when given 30 minutes. Table \ref{TabI1All30} shows that for instance set 1, 70\% of all schedules were solved to optimality. Even though none of the schedules for network five were solved to optimality, every schedule we tested was solved to within 0.3\% of optimality by both methods. LP-R solved one more schedule in total than MAIN, however MAIN was more consistent in closing the optimality gap. The extra run time resulted in a tighter gap and four extra schedules being solved to optimality in instance set 2, as seen in Table \ref{TabI2All30}. The average gap was now less than 1.6\% for LP-R and 1.3\% for MAIN. The maximum gap had dropped from 9.2\% to 3.3\% for LP-R, and all solutions for MAIN achieved a gap of less than 2.6\%.

\begin{table}
	\caption{Average and maximum run times, average, minimum and maximum optimality gap and number of schedules solved to optimality for instance set 1 after 30 minutes}
	\label{TabI1All30}
	\begin{tabular*}{\textwidth}{@{\extracolsep{\fill}}cccccccccc}
		\hline & & 1 & 2 & 3 & 4 & 5 & 6 & 7 & 8 \\ 
		\hline \multirow{3}{*}{MAIN} & avg. time (s) & 13.3 & 41.3 & 12.3 & 1169.3 & 1800.7 & 357.9 & 1371.3 & 542.6 \\ 
		& max time (s) & 31.0 & 58.0 & 23.0 & 1800.0 & 1803.0 & 1482.0 & 1801.0 & 1801.0 \\ 
		& avg. gap (\%) & 0.000 & 0.000 & 0.000 & 0.011 & 0.172 & 0.000 & 0.047 & 0.004 \\ 
		& max gap (\%) & 0.000 & 0.000 & 0.000 & 0.048 & 0.292 & 0.000 & 0.223 & 0.024 \\ 
		& num solved & 10 & 10 & 10 & 4 & 0 & 10 & 3 & 8 \\ 
		& avg. nodes & 1333.0 & 2528.0 & 912.7 & 79959.9 & 24851.9 & 21291.0 & 25772.4 & 11528.5 \\ 
		\hline \multirow{3}{*}{LP-R} & avg. time (s) & 14.1 & 38.8 & 13.2 & 1070.6 & 1801.9 & 457.3 & 1346.8 & 524.9 \\ 
		& max time (s) & 30.0 & 62.0 & 18.0 & 1800.0 & 1819.0 & 1800.0 & 1827.0 & 1800.0 \\ 
		& avg. gap (\%) & 0.000 & 0.000 & 0.000 & 0.011 & 0.182 & 0.000 & 0.049 & 0.007 \\ 
		& max gap (\%) & 0.000 & 0.000 & 0.000 & 0.059 & 0.293 & 0.003 & 0.266 & 0.055 \\ 
		& num solved & 10 & 10 & 10 & 6 & 0 & 9 & 3 & 8 \\ 
		& avg. nodes & 586.6 & 1957.0 & 74.0 & 58234.4 & 23513.9 & 31531.1 & 20548.7 & 11093.7 \\ 
		\hline \multirow{3}{*}{Direct MIP} & avg. time (s) & 854.2 & 1804.1 & 44.7 & 1808.1 & 1813.8 & 1808.5 & 1816.1 & 1822.4 \\ 
		& max time (s) & 1803.0 & 1805.0 & 221.0 & 1809.0 & 1837.0 & 1810.0 & 1817.0 & 1824.0 \\ 
		& avg. gap (\%) & 0.020 & 0.369 & 0.000 & 0.661 & 1.370 & 2.222 & 0.306 & 0.187 \\ 
		& max gap (\%) & 0.106 & 0.604 & 0.000 & 1.172 & 1.885 & 3.868 & 0.628 & 0.589 \\ 
		& num solved & 6 & 0 & 10 & 0 & 0 & 0 & 0 & 0 \\ 
		& avg. nodes & 85535.7 & 41985.7 & 1202.6 & 20318.8 & 7407.8 & 20615.8 & 7226.5 & 5877.6 \\ 
		& num. LR Solved& 10.0 & 10.0 & 10.0 & 10.0 & 10.0 & 10.0 & 10.0 & 10.0 \\ 
		\hline 
	\end{tabular*} 
\end{table}

\begin{table}
	\caption{Average and maximum run times, average, minimum and maximum optimality gap and number of schedules solved to optimality for instance set 2 after 30 minutes}
	\label{TabI2All30}
	\begin{tabular*}{\textwidth}{@{\extracolsep{\fill}}cccccccccc}
		\hline & & 1 & 2 & 3 & 4 & 5 & 6 & 7 & 8 \\ 
		\hline \multirow{3}{*}{MAIN} & avg. time (s) & 40.8 & 1656.6 & 12.7 & 1470.7 & 1802.0 & 1800.7 & 1663.0 & 651.4 \\ 
		& max time (s) & 103.0 & 1800.0 & 14.0 & 1845.0 & 1811.0 & 1804.0 & 1863.0 & 1800.0 \\ 
		& avg. gap (\%) & 0.000 & 0.163 & 0.000 & 0.188 & 1.298 & 1.067 & 0.369 & 0.074 \\ 
		& max gap (\%) & 0.000 & 0.462 & 0.000 & 0.502 & 2.525 & 1.842 & 1.022 & 0.514 \\ 
		& num solved & 10 & 2 & 10 & 2 & 0 & 0 & 0 & 8 \\ 
		& avg. nodes & 2731.0 & 58544.3 & 0.0 & 17895.8 & 7161.6 & 8203.7 & 15450.2 & 3731.5 \\ 
		\hline \multirow{3}{*}{LP-R} & avg. time (s) & 46.1 & 1734.6 & 14.0 & 1459.6 & 1800.6 & 1800.9 & 1803.9 & 610.5 \\ 
		& max time (s) & 109.0 & 1800.0 & 16.0 & 1803.0 & 1804.0 & 1805.0 & 1825.0 & 1871.0 \\ 
		& avg. gap (\%) & 0.000 & 0.170 & 0.000 & 0.209 & 1.565 & 0.971 & 0.275 & 0.078 \\ 
		& max gap (\%) & 0.000 & 0.512 & 0.000 & 0.639 & 3.321 & 1.711 & 1.279 & 0.537 \\ 
		& num solved & 10 & 2 & 10 & 2 & 0 & 0 & 0 & 8 \\ 
		& avg. nodes & 3357.1 & 59355.0 & 0.0 & 19754.0 & 7798.0 & 8691.1 & 19436.4 & 2780.5 \\ 
		\hline \multirow{3}{*}{Direct MIP} & avg. time (s) & 1629.5 & 1804.3 & 30.4 & 1709.9 & 1811.9 & 1809.2 & 1817.3 & DNS \\ 
		& max time (s) & 1803.0 & 1806.0 & 43.0 & 1847.0 & 1812.0 & 1810.0 & 1818.0 & DNS \\ 
		& avg. gap (\%) & 0.100 & 2.244 & 0.000 & 0.761 & 3.505 & 5.968 & 153.276 & 100.000 \\ 
		& max gap (\%) & 0.381 & 4.179 & 0.000 & 1.828 & 6.746 & 8.526 & 1365.990 & 100.000 \\ 
		& num solved & 1 & 0 & 10 & 1 & 0 & 0 & 0 & 0 \\ 
		& avg. nodes & 97099.2 & 17865.8 & 84.4 & 8530.0 & 1308.6 & 5127.8 & 376.4 & 0.0 \\ 
		& num. LR Solved& 10.0 & 10.0 & 10.0 & 10.0 & 10.0 & 10.0 & 9.0 & 0.0 \\ 
		\hline 
	\end{tabular*} 
\end{table}

While not always the best, our MAIN implementation performed consistently well compared to other methods, especially for the large networks. We also compared our method to the direct MIP formulation in Gurobi. This is for comparison with state-of-the-art general purpose solvers. Gurobi is effective for solving linear programs and small integer programs, though larger integer problems may benefit from a tailored approach, such as the one we describe. What is clear is that our method has better scalability than the direct MIP implementation. For the first four networks, Gurobi performs comparably, however as the networks get larger, the direct MIP fails to solve even the linear relaxation in 30 minutes.

Table \ref{TabI1All30} shows that the direct implementation had difficulty in solving many of the cases. With the exception of network 3, the average gap for the MIP was always at least six times larger than for LP-R. Table \ref{TabI2All30} shows how the MIP struggled with the larger networks, and again with the exception of network 3, fewer cases were solved to optimality than by MAIN and LP-R. In instance set 2, the direct implementation failed to solve any of the linear relaxations of the job schedules of network 8.

Network 3 is an interesting case because of the number of nodes explored. In general, Benders decomposition yields a more relaxed problem. This means that the branch-and-bound tree should be larger, which is acceptable because the nodes can be explored much more quickly. However, in the case of network 3, especially for instance set 2, the number of nodes explored by our implementations is fewer than those explored by the direct MIP. This is because the objective value of the solution to the relaxed problem is equal to or very close to the IP solution's objective value. The result of this is that Gurobi must extract an integer solution with a similar objective value, and it is easier to do this for smaller models, which is why fewer nodes are explored by our implementations. For instance set 2, the solver-added cuts are sufficient to find a solution at the root node, so that branch-and-bound is not required.

Moreover, in most instances we found that Gurobi is able to add cuts at the root node much more effectively in the MAIN or LP-R models, compared to the Direct MIP. For example, in one instance the direct implementation started the Branch and Bound process after 900s with an optimality gap of 3.51\%, while the MAIN implementation started after only 300s with a gap of 2.07\%. This means that not only do we generate speed improvements by being able to search the Branch and Bound tree more quickly, the number of nodes in the Branch and bound tree also decreases due to the reduced model size and the improved bounds found by Gurobi's cutting algorithms.

This can be seen in Tables \ref{TabI1All30} and \ref{TabI2All30} from the average number of Branch and Bound nodes explored. In cases where all problems solved to optimality, our implementations required fewer nodes to be explored due to the tighter bounds and solver-added cuts. In the larger cases where no problems were solved to optimality, our implementations explored more nodes due to the increased efficiency resulting from disaggregated Benders decomposition.

\begin{table}
	\caption{Average and maximum run times and average number of branch-and-bound nodes for instance set 3}
	\label{TabI3Time}
	\begin{tabular*}{\textwidth}{@{\extracolsep{\fill}}cccccccccc}
		\hline & & 1 & 2 & 3 & 4 & 5 & 6 & 7 & 8 \\ 
		\hline \multirow{3}{*}{DBD} & avg. time (s) & 3.8 & 5.4 & 6.1 & 12.7 & 18.5 & 12.7 & 36.2 & 37.9 \\
		& max time (s) & 4 & 6 & 7 & 20 & 20 & 14 & 58 & 44 \\
		& avg. nodes & 251.1 & 61.5 & 0.6 & 416.0 & 191.8 & 41.3 & 1760.9 & 97.1 \\
		\hline \multirow{3}{*}{Pre-cuts} & avg. time (s) & 4.0 & 6.5 & 7.2 & 13.9 & 19.4 & 16.5 & 30.6 & 41.2 \\ 
		& max time (s) & 4 & 7 & 8 & 15 & 22 & 19 & 42 & 44 \\ 
		& avg. nodes & 0.4 & 0.4 & 1.0 & 15.2 & 125.0 & 0.0 & 652.9 & 6.7 \\
		\hline \multirow{3}{*}{MAIN} & avg. time (s) & 4.0 & 6.4 & 7.2 & 14.1 & 19.4 & 16.6 & 30.7 & 41.5 \\
		& max time (s) & 4 & 7 & 8 & 15 & 22 & 19 & 42 & 45 \\
		& avg. nodes & 0.4 & 0.4 & 1.0 & 15.1 & 125.3 & 0.0 & 591.7 & 6.7 \\
		\hline \multirow{4}{*}{LP-R} & avg. time (s) & 6.0 & 8.8 & 10.1 & 18.37 & 25.4 & 20.3 & 38.2 & 54.0 \\
		& max time (s) & 7 & 9 & 11 & 19 & 29 & 23 & 42 & 60 \\
		& avg. nodes & 0.2 & 0.0 & 0.0 & 3.7 & 4.3 & 0.0 & 110.7 & 9.3 \\
		& avg. LP solves & 3.9 & 4.2 & 4.1 & 4.8 & 5.0 & 4.1 & 5.1 & 5.4 \\
		\hline 
	\end{tabular*}
\end{table}

Table \ref{TabI3Time} shows the average and maximum run times for the problems in instance set 3. All instances in this set were solved to optimality in less than one minute. Because these instances are relatively easy to solve, there is not much difference between our implementations. Solving the linear relaxation first does not benefit these cases because they are easy enough to solve without it. The biggest difference is in the number of nodes explored. The pre-processing of bottlenecks greatly improves the initial bounds and results in fewer nodes being explored.

Another common difference is the number of lazy constraints added during the solution process. Typically DBD has the most and LP-R has the fewest. This is because when we add the initial cuts, as in Pre-cuts, we no longer have to add them later as lazy constraints. The same applies to LP-R: more cuts are added as hard constraints before we begin the branch-and-bound process, so fewer lazy constraints are added.

\subsection{Real-world Instances}
We also tested our techniques on the instance sets derived from real-world data, provided by Boland et. al. \cite{Boland2013}. The network is representative of the Hunter Valley Coal Chain, and there are two lists of jobs designed to span one year each, for 2010 and 2011. The time is discretised into hours, so there are 8761 time periods for each problem (since 2010 and 2011 are not leap years). The unrestricted flow over one year is 161.3 Mt, and the success of an algorithm is measured in the minimisation of the impact on the network. The majority of jobs have durations between 1 and 18 hours, while the potential time window is set at two weeks for each job.

The fact that the potential job window is significantly larger than the duration of the job in all cases leads to the problem described in Section \ref{LPSubSect}. The solver tends to start with a very weak LP bound and takes more time to converge to a solution. Table \ref{TabRWGapC} shows that after two hours, the optimality gap is 7.4\% $\pm$ 2\% in all cases. It is possible to tell the solver to focus on finding better solutions rather than lowering the bound or proving optimality. In this case, the optimality gap is similar to before, however the objective values are higher and comparable to the results in Boland et. al. It is also possible to tell the solver not to add its own cuts and to instead use only the constraints we have added. The main benefit of this is that we enter the branch-and-bound process much earlier and thus can explore more nodes. The drawback is that the upper bound is not as tight as it would have been had the solver  added its own cuts.

\begin{table}
	\caption{Comparison between MAIN and LP-R, with and without solver-added cuts, with regular or objective-value-focussed solve priorities on 2010 data after two hours. All numbers are in Mt.}
	\label{TabRWGapC}
	\begin{tabular*}{\textwidth}{@{\extracolsep{\fill}}cccccc}
		\hline & & \multicolumn{2}{c}{Standard} & \multicolumn{2}{c}{Focussed} \\
		& & Objective & Bound & Objective & Bound \\
		\hline \multirow{2}{*}{Cuts} & MAIN & 132.3 & 141.1 & 135.7 & 143.1 \\
		& LP-R & 132.5 & 141.1 & 131.5 & 142.2 \\
		\multirow{2}{*}{No Cuts} & MAIN & 133.0 & 145.8 & 135.4 & 144.2 \\
		& LP-R & 135.0 & 143.8 & 135.3 & 143.3 \\		
	\end{tabular*}
\end{table}

\begin{table}
	\caption{Comparison between MAIN and LP-R, with and without solver-added cuts, with regular or objective-value-focussed solve priorities on 2011 data after two hours. All numbers are in Mt.}
	\label{TabRWGapNC}
	\begin{tabular*}{\textwidth}{@{\extracolsep{\fill}}cccccc}
		\hline & & \multicolumn{2}{c}{Standard} & \multicolumn{2}{c}{Focussed} \\
		& & Objective & Bound & Objective & Bound \\
		\hline \multirow{2}{*}{Cuts} & MAIN & 138.0 & 145.5 & 139.0 & 146.3 \\
		& LP-R & 136.3 & 145.6 & 138.9 & 146.1 \\
		\multirow{2}{*}{No Cuts} & MAIN & 136.9 & 149.2 & 139.0 & 148.0 \\
		& LP-R & 139.0 & 147.4 & 139.9 & 147.4 \\
	\end{tabular*}
\end{table}

When looking across years and solver focus, in three out of four cases LP-R found the highest objective value, which is within 1.6 Mt of the best found solution in Boland et. al. for both instance sets. In most cases focussing the solver towards finding better solutions does yield better objective values, however the upper bound suffers as a result. Removing solver-added cuts does not yield much improvement when using MAIN, whereas LP-R finds objective values between 1 and 4 Mt higher in every case, regardless of the solver focus. This is likely because the initial cuts added during the pre-cuts stage and the additional constraints added from the relaxed problem are sufficient, and less time is spent at the root node adding cuts, so the branch-and-bound process is started sooner.

While these methods cannot prove optimality in two hours, neither can a heuristic or a direct implementation in CPLEX or Gurobi. The heuristics used in Boland et. al. are very good and our methods return similar objective values to theirs. If the solver focus were switched to providing a better upper bound, it may make a good bounding tool for checking the heuristic solutions against. Indeed, one can imagine a hybrid solution technique where two algorithms are run simultaneously, on separate threads or separate computers.

\section{Conclusion}
We have shown that disaggregated Benders decomposition is an effective technique for solving the MaxTFFAO problem. In many simulated cases optimal solutions can be found in a short enough time to be practically useful. The amount of choice in the real world problems makes it difficult to prove optimality, however reasonable solutions can be returned in a short amount of time. In conjunction with the advanced capabilities of solvers such as Gurobi, disaggregated Benders decomposition can result in smaller problems with tighter bounds and smaller branch-and-bound trees. In future we would like to consider more broadly the class of integrated network design and scheduling problems to which this technique also applies. It would also be interesting to look at other problems to which disaggregated Benders decomposition can be applied, and see if similar benefits can be obtained.

\section*{References}

\bibliography{DAM-Paper}

\begin{thebibliography}{10}
\expandafter\ifx\csname url\endcsname\relax
  \def\url#1{\texttt{#1}}\fi
\expandafter\ifx\csname urlprefix\endcsname\relax\def\urlprefix{URL }\fi
\expandafter\ifx\csname href\endcsname\relax
  \def\href#1#2{#2} \def\path#1{#1}\fi

\bibitem{Boland2013}
H.~W. N.~Boland, T.~Kalinowski, L.~Zheng, Scheduling arc maintenance jobs in a
  network to maximize total flow over time, Discrete Applied Mathematics 163
  (2014) 34--52.
\newblock \href {http://dx.doi.org/10.1016/j.dam.2012.05.027}
  {\path{doi:10.1016/j.dam.2012.05.027}}.

\bibitem{Nurre2012}
J.~M. T.~S. S.G.~Nurre, B.~Cavdaroglu, W.~Wallace, Restoring infrastructure
  sytems: An integrated network design and scheduling (inds) problem, European
  Journal of Operational Research 223 (2012) 794--806.
\newblock \href {http://dx.doi.org/10.1016/j.ejor.2012.07.010}
  {\path{doi:10.1016/j.ejor.2012.07.010}}.

\bibitem{Tang2013}
W.~J. L.~Tang, G.~Saharidis, An improved benders decomposition algorithm for
  the logistics facility location problem with capacity expansions, Annals of
  Operations Research 210 (2013) 165--190.
\newblock \href {http://dx.doi.org/10.1007/s10479-011-1050-9}
  {\path{doi:10.1007/s10479-011-1050-9}}.

\bibitem{Benders1962}
J.~Benders, Partitioning procedures for solving mixed-variables programming
  problems, Numerische Mathematik 4 (1962) 238--252.

\bibitem{Camargo2008}
G.~M.~J. R.S.~de Camargo, H.~Luna, Benders decomposition for the uncapacitated
  multiple allocation hub location problem, Computers and Operations Research
  35 (2008) 1047--1064.
\newblock \href {http://dx.doi.org/10.1016/j.cor.2006.07.002}
  {\path{doi:10.1016/j.cor.2006.07.002}}.

\bibitem{Lusby2013}
L.~M. R.~Lusby, B.~Petersen, A solution approach based on benders decomposition
  for the preventive maintenance scheduling problem of a stochastic large-scale
  energy system, Journal of Scheduling 16 (2013) 605--628.
\newblock \href {http://dx.doi.org/10.1007/s10951-012-0310-0}
  {\path{doi:10.1007/s10951-012-0310-0}}.

\bibitem{Elias1956}
A.~F. P.~Elias, C.~Shannon, A note on the maximum flow through a network, IRE
  Transactions on Information Theory IT-2 (1956) 117--119.

\bibitem{Ford1956}
L.~F. Jr., D.~Fulkerson, Maximal flow through a network, Canadian Journal of
  Mathematics 8 (1956) 399--404.

\bibitem{Magnanti1981}
T.~Magnanti, R.~Wong, Accelerating benders decomposition: Algorithmic
  enhanvement and model selection criteria, Operations Research 29~(3) (1981)
  464--484.

\bibitem{Magnanti1984}
T.~Magnanti, R.~Wong, Network design and transportation planning: Models and
  algorithms, Transportation Science 18~(1) (1984) 1--55.

\bibitem{gurobi}
I.~Gurobi~Optimization, \href{http://www.gurobi.com}{Gurobi optimizer reference
  manual} (2015).
\newline\urlprefix\url{http://www.gurobi.com}

\end{thebibliography}

\end{document}